\documentclass[reqno]{amsart}
\usepackage{hyperref}
\usepackage{amssymb}

\allowdisplaybreaks

\theoremstyle{plain}
\newtheorem{theorem}{Theorem}[section]

\newtheorem{lemma}[theorem]{Lemma}

\theoremstyle{definition}
\newtheorem{definition}[theorem]{Definition}

\theoremstyle{remark}
\newtheorem{remark}[theorem]{Remark}
\numberwithin{equation}{section}

\begin{document}

\title[Nonlocal parabolic equation with nonlocal boundary condition]
{Global existence of solutions of initial-boundary value problem for nonlocal
      parabolic equation with nonlocal boundary condition}

\author[A. Gladkov]{Alexander Gladkov}
\address{Alexander Gladkov \\ Department of Mechanics and Mathematics
\\ Belarusian State University \\  4  Nezavisimosti Avenue \\ 220030
Minsk, Belarus  and  Peoples' Friendship University of Russia (RUDN University) \\  6 Miklukho-Maklaya street \\  117198 Moscow,  Russian Federation}    \email{gladkoval@mail.ru }

\author[T. Kavitova]{Tatiana Kavitova}
\address{Tatiana Kavitova \\ Department of Mathematics and Information Technologies, Vitebsk State University,
33 Moskovskii pr., Vitebsk, 210038, Belarus}
\email{kavitovatv@tut.by}
 

\begin{abstract}
We prove global existence and blow-up of solutions of initial-boundary value problem
for nonlinear nonlocal parabolic equation with nonlinear nonlocal boundary condition. Obtained
results depend on the behavior of variable coefficients for large values of time.
\end{abstract}

\maketitle

\section{Introduction}\label{in}
We consider nonlinear nonlocal parabolic equation
        \begin{equation}\label{v:u}
            u_t=\Delta u+a(x,t)u^r\int_\Omega{u^p(y,t)}\,dy-b(x,t)u^q,\;x\in\Omega,\;t>0,
        \end{equation}
with nonlinear nonlocal boundary condition
        \begin{equation}\label{v:g}
        u(x,t)=\int_{\Omega}{k(x,y,t)u^l(y,t)}\,dy,\;x\in\partial\Omega,\;t>0,
        \end{equation}
and initial datum
        \begin{equation}\label{v:n}
            u(x,0)=u_{0}(x),\; x\in\Omega,
        \end{equation}
where $r,\,p,\,q,\,l$ are positive constants, $\Omega$  is a bounded domain in $\mathbb{R}^n$
for $n\geq1$ with smooth boundary $\partial\Omega.$

Throughout this paper we suppose that $a(x,t),\;b(x,t),\;k(x,y,t)$ and $u_0(x)$ satisfy the following conditions:
        \begin{equation*}
        a(x,t),\;b(x,t)\in C^\alpha_{loc}(\overline{\Omega}\times[0,\infty)),\;0<\alpha<1,\;a(x,t)\geq0,\;b(x,t)\geq0;
        \end{equation*}
        \begin{equation*}
            k(x,y,t)\in C(\partial\Omega\times\overline{\Omega}\times[0,\infty)),\;k(x,y,t)\geq0;
        \end{equation*}
        \begin{equation*}
            u_0(x)\in C(\overline{\Omega}),\;u_0(x)\geq0,\;x\in\overline\Omega,\;u_0(x)=\int_{\Omega}{k(x,
            y,0)u_0^l(y)}\,dy,\;x\in\partial\Omega.
        \end{equation*}

For global existence and blow-up of solutions for parabolic equations
with nonlocal boundary conditions we refer to \cite{A,CL,D}, \cite{F}--\cite{GK2}, \cite{GK,KT,K,MV,P,WMX1} and the references therein.
Initial-boundary value problems  for nonlocal parabolic equations
with nonlocal boundary conditions were considered in many papers also
(see, for example, \cite{CY,FZ1,FZ3,FZY,LL,L,WMX2}). In particular, blow-up problem  for nonlocal parabolic equations
with  boundary condition (\ref{v:g}) was investigated in \cite{CYZ,FZ2,LMM,LM,LMA,MLZ,ZT,ZY}. So, for example, the authors of \cite{CYZ}  studied (\ref{v:u})--(\ref{v:n}) with $b(x,t) \equiv 0,\,$  $a(x,t) \equiv a(x)\,$ and  $k(x,y,t) \equiv
k(x,y),$ and problem (\ref{v:u})--(\ref{v:n})
with  $r=0, \,$ $a(x,t) \equiv 1,\,$ $b(x,t) \equiv b>0\,$ and $k(x,y,t)
\equiv k(x,y)$ was considered in \cite{MLZ}.  The authors of \cite{GG1} studied
(\ref{v:u})--(\ref{v:n}) with  $a(x,t) \equiv 0.$

The existence of classical local solutions and comparison principle for (\ref{v:u})--(\ref{v:n}) were proved in \cite{GK3} and \cite{GK4}.

In this paper we prove global existence and blow-up of solutions of~(\ref{v:u})--(\ref{v:n}).
Obtained results depend on the behavior of variable coefficients
 $a(x,t) ,\,$ $b(x,t)\,$  and  $k(x,y,t)$ as $t \to \infty.$

This paper is organized as follows. Global existence of solutions
for any initial data and blow-up in finite time of solutions for
large initial data are proved in section 2. In section 3 we
present finite time blow-up of all nontrivial solutions as well as
the existence of global solutions for small initial data.

%%%%%%%%%%%%%%%%%%%%%%%%%%%%%%%%%%%%%%%%%%%%%%%%%%%%%%%%%%%%%%%%%%%%%%%%%%%%%%%%%%%%%%%%%%%
\section{Global existence and blow-up of solutions}\label{gl}
Let $Q_T=\Omega\times(0,T),\;S_T=\partial\Omega\times(0,T),$
$\Gamma_T=S_T\cup\overline\Omega\times\{0\},$ $T>0.$
\begin{definition}\label{Gl:sup}
 We say that a nonnegative function $u(x,t)\in C^{2,1}(Q_T)\cap
    C(Q_T\cup\Gamma_T)$ is a supersolution
    of~(\ref{v:u})--(\ref{v:n}) in $Q_{T}$ if
    \begin{equation}\label{Gl:sup^u}
    u_t\geq\Delta u+a(x,t)u^r\int_\Omega{u^p(y,t)}\,dy-b(x,t)u^q,\;(x,t)\in Q_T,
    \end{equation}
    \begin{equation}\label{Gl:sup^g}
    u(x,t)\geq\int_{\Omega}{k(x,y,t)u^l(y,t)}\,dy,\;(x,t)\in S_T,
    \end{equation}
    \begin{equation}\label{Gl:sup^n}
    u(x,0)\geq u_{0}(x),\; x\in\Omega,
    \end{equation}
and  $u(x,t)\in C^{2,1}(Q_T)\cap
    C(Q_T\cup\Gamma_T)$ is a subsolution of~(\ref{v:u})--(\ref{v:n}) in $Q_{T}$ if
    $u\geq0$ and it satisfies~(\ref{Gl:sup^u})--(\ref{Gl:sup^n}) in the reverse
order. We say that $u(x,t)$ is a solution of~(\ref{v:u})--(\ref{v:n}) in $Q_T$ if $u(x,t)$
    is both a subsolution and a supersolution of~(\ref{v:u})--(\ref{v:n}) in $Q_{T}.$
\end{definition}

We will repeatedly use the following comparison principle (see \cite{GK3}, \cite{GK4}).
\begin{theorem}\label{Gl:theorem:comp-prins}
   Let $\underline u(x,t)$ and $\overline u(x,t)$ be a subsolution and a supersolution
   of problem~(\ref{v:u})--(\ref{v:n}) in $Q_{T},$ respectively. Suppose that $\underline u(x,t)>0$ or
   $\overline u(x,t)>0$ in $Q_T\cup\Gamma_T$ if $\min(r,p,l)<1.$ Then $\overline u(x,t)\geq
    \underline u(x,t)$  in $Q_T\cup\Gamma_T.$
\end{theorem}
To prove global existence of solutions of~(\ref{v:u})--(\ref{v:n}) we suppose that
\begin{equation}\label{Gl:b(x,t)}
b(x,t)> 0\textrm{ for }x\in\overline\Omega\textrm{ and } t \geq 0.
\end{equation}
\begin{theorem}\label{Gl:theorem1}
    Let $\max{(r+p,l)} \leq 1$ or (\ref{Gl:b(x,t)}) hold and either $l \leq 1, \, 1 < r+p < q$
    or $1<l<(q+1)/2,\, \max (r+p, 2p+1) < q.$
      Then problem~(\ref{v:u})--(\ref{v:n}) has global solutions for any initial data.
\end{theorem}
\begin{proof}
   Let $T$ be any positive constant and
    \begin{equation}\label{Gl:M}
    M=\max\left(\sup\limits_{Q_T}{a(x,t)},\,\sup\limits_{\partial\Omega\times Q_T}{k(x,y,t)}\right).
    \end{equation}
     In order to prove global existence of solutions we construct a
suitable explicit supersolution of~(\ref{v:u})--(\ref{v:n}) in
$Q_T.$

   Suppose at first that $\max(r+p,l)\leq1.$ Let $\lambda_1$ be the first
eigenvalue of the following problem
    \begin{equation}\label{Bl:phi}
     \Delta\varphi (x) + \lambda\varphi (x) = 0,\; x\in\Omega, \quad
    \varphi(x)=0,\;x\in\partial\Omega,
    \end{equation}
and $\varphi(x)$ be the corresponding eigenfunction which is
chosen to satisfy that for some $ 0 <\varepsilon <1 $
    \begin{equation*}
    M\int_{\Omega}{\frac{dy}{(\varphi(y)+\varepsilon)^l}} \leq 1.
    \end{equation*}
Then it is easy to check that
    \begin{equation}\label{Gl:v}
    v(x,t)=\frac{\eta\exp{(\mu t)}}{\varphi(x)+\varepsilon}
    \end{equation}
   is a supersolution of~(\ref{v:u})--(\ref{v:n}) in $Q_T$ if
    $$
    \eta\geq\max\left(\sup_\Omega{u_0(x)}\sup_\Omega{\left(\varphi(x)+\varepsilon\right)},1\right),
    $$
    $$
    \mu\geq\lambda_1+\sup_\Omega\frac{2|\nabla\varphi(x)|^2}{\left(\varphi(x)+\varepsilon\right)^2}+
    M\sup_\Omega{(\varphi(x)+\varepsilon)^{1-r}}\int_\Omega{\frac{dy}{\left(\varphi(y)+\varepsilon\right)^p}}.
    $$
By Theorem~\ref{Gl:theorem:comp-prins} problem
(\ref{v:u})--(\ref{v:n}) has global solutions for any initial data.

 From~(\ref{Gl:b(x,t)}) we conclude that
$\underline b = \inf\limits_{Q_T}{b(x,t)} > 0.$ Let $l \leq 1,\;1< r+p <q.$
Then $v(x,t)$ in (\ref{Gl:v}) is a supersolution of (\ref{v:u})--(\ref{v:n}) in $Q_T$ if
    $$
    \eta\geq\max\left(\sup_\Omega{u_0(x)}\sup_\Omega{\left(\varphi(x)+\varepsilon\right)},
    \left(\frac{M}{\underline b}\sup_\Omega{(\varphi(x)+\varepsilon)^{q-r}}\int_\Omega{\frac{dy}
    {\left(\varphi(y)+\varepsilon\right)^p}}\right)^{\frac{1}{q-r-p}}, 1\right),
    $$
    $$
    \mu\geq\lambda_1+\sup_\Omega\frac{2|\nabla\varphi(x)|^2}{\left(\varphi(x)+\varepsilon\right)^2}.
    $$

Let $1<l<(q+1)/2,\, \max (r+p, 2p+1) < q.$ To construct a  supersolution
we use the change of variables in a
neighborhood of $\partial \Omega$ as in \cite{CPE}. Let
$\overline x\in\partial \Omega$ and $\widehat{n}
(\overline x)$ be the inner unit normal to $\partial \Omega$ at the
point $\overline x.$ Since $\partial \Omega$ is smooth it is well
known that there exists $\delta >0$ such that the mapping $\psi
:\partial \Omega \times [0,\delta] \to \mathbb{R}^n$ given by
$\psi (\overline x,s)=\overline x +s\widehat{n} (\overline x)$
defines new coordinates ($\overline x,s)$ in a neighborhood of
$\partial \Omega$ in $\overline\Omega.$ A straightforward
computation shows that, in these coordinates, $\Delta$ applied to
a function $g(\overline x,s)=g(s),$ which is independent of the
variable $\overline x,$ evaluated at a point $(\overline x,s)$ is
given by
     \begin{equation}\label{Gl:new-coord}
     \Delta g(\overline x,s)=\frac{\partial^2g}{\partial s^2}(\overline x,s)-\sum_{j=1}^{n-1}\frac{H_j(\overline x)}{1-s
        H_j (\overline x)}\frac{\partial g}{\partial s}(\overline x,s),
     \end{equation}
where $H_j (\overline x)$ for $j=1,...,n-1,$ denotes the principal
curvatures of $\partial\Omega$ at $\overline x.$ For $0\leq s\leq \delta$
and small $\delta$  we have
     \begin{equation}\label{Gl:enq4}
     \left|\sum_{j=1}^{n-1} \frac{H_j (\overline x)}{1-s H_j (\overline
        x)}\right|\leq\overline c.
     \end{equation}

Let $0<\varepsilon<\omega<\min(\delta, 1), $
$2/(q-1) <\beta< \min(1/p, 1/(l-1)),$  $0<\gamma<\beta/2,$ $A\ge\sup_\Omega u_0(x).$
For points in $Q_{\delta,T}=\partial \Omega \times [0,
\delta]\times [0,T]$ of coordinates $(\overline x,s,t)$ define
     \begin{equation}\label{Gl:function}
    v(x,t)=  v((\overline x,s),t)=\left[(s+\varepsilon)^{-\gamma}-\omega^{-\gamma}\right]_+^\frac{\beta}{\gamma}+A,
     \end{equation}
where $s_+=\max(s,0).$ For points in $\overline{Q_T}\setminus Q_{\delta,T}$
we set  $ v(x,t)= A.$ We prove that $ v(x,t)$
is a supersolution of (\ref{v:u})--(\ref{v:n}) in $Q_T.$
It is not difficult to check that
     \begin{equation}\label{Gl:enq1}
     \left|\frac{\partial v}{\partial s}\right|\leq \beta\min\left(\left[ D(s)\right]^\frac{\gamma+1}{\gamma}\left[( s+\varepsilon)^{-\gamma}-\omega^{-\gamma}\right]_+^\frac{\beta+1}{\gamma},\,(s+\varepsilon)^{-(\beta+1)}\right),
     \end{equation}
     \begin{equation}\label{Gl:enq2}
     \left|\frac{\partial^2 v}{\partial s^2}\right|\leq \beta (\beta+1)\min\left(\left[D(s)\right]^\frac{2(\gamma+1)}{\gamma}\left[( s+\varepsilon)^{-\gamma}-
     \omega^{-\gamma}\right]_+^\frac{\beta+2}{\gamma},\,( s+\varepsilon)^{-(\beta+2)}\right),
     \end{equation}
 where
     $$
     D(s)= \frac{( s+\varepsilon)^{-\gamma}}{ (s+\varepsilon)^{-\gamma}-\omega^{-\gamma}}.
     $$
Then $D^\prime(s)>0$ and for any $\overline\varepsilon>0$
     \begin{equation}\label{Gl:enq3}
     1\leq D(s)\leq 1+\overline\varepsilon, \; 0<s\leq{\overline s},
     \end{equation}
where ${\overline s} = [\overline\varepsilon/(1+\overline\varepsilon)]^{1/\gamma}\omega-\varepsilon,$
     $\varepsilon<[\overline\varepsilon/(1+\overline\varepsilon)]^{1/\gamma}\omega.$
We denote
\begin{equation}\label{L}
Lv \equiv v_t-\Delta
v-a(x,t)v^r\int_{\Omega}{v^p(y,t)}\,dy+b(x,t)v^q
    \end{equation}
    and
\begin{equation}\label{Gl:J1}
                 \overline J =\sup_{0<s<\delta}\int_{\partial\Omega}|J(\overline y,s)|\,d\overline y,
        \end{equation}
where $J(\overline y,s)$ is Jacobian of the change of variables in neighborhood of $\partial\Omega.$
We use the inequality $(a+b)^p \leq 2^p (a^p + b^p), \, a\ge 0, b\ge 0, p > 0$ to estimate the integral in (\ref{L})
 \begin{eqnarray}\label{Est}
\int_{\Omega}{v^p(y,t)}\,dy & \leq & 2^p \left( A^p |\Omega| +
\int_{0}^{\omega - \varepsilon}  \int_{\partial\Omega}
J(\overline y,s) \left[(s+\varepsilon)^{-\gamma}-\omega^{-\gamma}\right]_+^\frac{\beta p}{\gamma} \, d\overline y \, ds \right)
 \nonumber \\
 & \leq &  2^p \left( A^p |\Omega| + \frac{\overline J \omega^{1-\beta p}}{1-\beta p} \right).
    \end{eqnarray}
Here $|\Omega|$ is Lebesque measure of $\Omega.$ By (\ref{Gl:new-coord})--(\ref{L}), (\ref{Est}) we can choose
$\overline\varepsilon$ small and  $A$ large so that in $Q_{{\overline s},T}$
   $$
 	Lv \geq  \underline b \left( \left[( s+\varepsilon)^{-\gamma}-\omega^{-\gamma}\right]_+^\frac{\beta }{\gamma} + A \right)^q - \beta(\beta+1)\left[ D(s)\right]^\frac{2(\gamma+1)}{\gamma}\left[( s+\varepsilon)^{-\gamma}-\omega^{-\gamma}\right]_+^\frac{\beta+2}{\gamma}
$$
$$
 	-  \beta\overline c\left[D(s)\right]^\frac{\gamma+1}{\gamma}
 \left[( s+\varepsilon)^{-\gamma}-\omega^{-\gamma}\right]_+^\frac{\beta+1}{\gamma}
$$
$$
 	-2^p M \left(\left[( s+\varepsilon)^{-\gamma}-\omega^{-\gamma}\right]_+^\frac{\beta}{\gamma}+A\right)^r
 \left( A^p |\Omega| + \frac{\overline J \omega^{1-\beta p}}{1-\beta p} \right) \geq 0.
$$
Let $s\in[{\overline s},\delta].$ From (\ref{Gl:new-coord})--(\ref{Gl:enq2}) we have
$$
|\Delta v| \leq \beta(\beta+1)\omega^{-(\beta+2)}\left(\frac{1+\overline\varepsilon}
{\overline\varepsilon}\right)^\frac{\beta+2}{\gamma} + \beta\overline c\omega^{-(\beta+ 1)}\left(\frac{1+\overline\varepsilon}{\overline\varepsilon}\right)^\frac{\beta+1}{\gamma}
$$
and by (\ref{Est}) $Lv\geq0$ for large $A.$ Obviously, in $\overline{Q_T}\setminus Q_{\delta,T}$
$$
Lv \geq - 2^p M A^r \left( A^p |\Omega| + \frac{\overline J \omega^{1-\beta p}}{1-\beta p} \right) + \underline b A^q \geq  0
$$
for large $A.$

Now we prove the following inequality
\begin{equation}\label{Gl:boundary}
     v((\overline x,0),t)\geq\int_{\Omega}{k(x,y,t)v^l(y,t)}\,dy, \: (x,t) \in S_T
\end{equation}
for a suitable choice of $\varepsilon.$ To do this we use the change of variables in neighborhood of
$\partial\Omega.$ Estimating the integral $I$ in the right hand side of~(\ref{Gl:boundary}), we get
       $$
 	I \leq 2^{l}M \overline J \int_{0}^{\omega - \varepsilon}
 \left[(s+\varepsilon)^{-\gamma}-\omega^{-\gamma}\right]_+^\frac{\beta l}{\gamma} \, ds + 2^l M A^l|\Omega|
 \leq 2^l M \overline J C(\varepsilon) + 2^{l}M A^l |\Omega|,
$$
where
      \begin{equation*}
    C(\varepsilon)=
        \begin{cases}
            \varepsilon^{-(\beta l -1)}/(\beta l -1),\; \beta l > 1,\\
          \omega^{1 - \beta l }/(1 - \beta l ),\; \beta l < 1,\\
            - \ln \varepsilon,\; \beta l = 1.
        \end{cases}
    \end{equation*}
 On the other hand, we have
     $$v ((\overline x,0),t)=\left[\varepsilon^{-\gamma}-\omega^{-\gamma}\right]_+^\frac{\beta}{\gamma}+A.$$
     Hence, (\ref{Gl:boundary}) holds for small values of $\varepsilon$ and by Theorem~\ref{Gl:theorem:comp-prins}
     $u(x,t)\leq v(x,t)$ in $\overline{Q_T}.$
\end{proof}

To prove finite time blow-up result we need lower bound for solutions of~(\ref{v:u})--(\ref{v:n}) with large initial data.

\begin{lemma}\label{Gl:Lem2}
Let $u(x,t)$ be a solution of~(\ref{v:u})--(\ref{v:n}) in $\overline{Q_T}.$
For any
$\Omega_0 \subset \subset \Omega_1 \subset\subset \Omega$
and any positive constant $C$ there exists positive constant
$c_1$ such that if $u_0(x)\geq c_1$ in $\Omega_1,$ then
 \begin{equation}\label{110}
	u(x,t)\geq C \textrm{ in }\overline{\Omega_0}\times[0,T].
\end{equation}
\end{lemma}
\begin{proof}
Let $y(x,t)$ be a solution of the following problem
 \begin{equation}\label{111}
\begin{cases}
	y_t=\Delta y,\;x\in\Omega_1,\;0<t<T,\\
	y(x,t)=0,\;x\in\partial\Omega_1,\;0<t<T,\\
	y(x,0)=\chi(x) ,\; x\in\Omega_1,
	\end{cases}
\end{equation}
where
$\chi(x)\in C_0^\infty(\Omega_1),$ $\chi(x)=1$ in
$\Omega_0$ and $0\leq\chi(x)\leq1.$
By strong maximum principle
\begin{equation}\label{112}
	\inf\limits_{\Omega_0\times(0,T)}{y(x,t)}>0.
\end{equation}
Suppose that
$q\geq1.$
We put
$
	m=\max\left(\sup_{Q_T}{u(x,t)},\sup_{Q_T}{b(x,t)}\right)
$
and define function
$v(x,t)=\exp{(\rho 	t)}u(x,t).$
For
$\rho\geq m^q$
we have in $Q_T$
$$
	v_t-\Delta v = \exp{(\rho t)}\left(\rho u+a(x,t)u^r\int_\Omega{u^p(y,t)}\,dy-b(x,t)u^q\right)
\geq v(\rho-b(x,t)u^{q-1}) \geq 0.
$$
We assume
$u_0(x)\geq c_1\chi(x)$
in
$\Omega_1,$
where constant $c_1$ will be chosen below. Then by comparison principle for
 (\ref{111}) we get
$v(x,t)\geq c_1y(x,t)$
in
$\overline{\Omega_1}\times[0,T].$
Taking into account~(\ref{112}), we have (\ref{110}) if
$c_1=C\exp{(\rho T)}\left(\inf_{\Omega_0\times(0,T)}{y(x,t)}\right)^{-1}.$

Let $q<1.$
We set $w(x,t)= \exp{(mt)} (u(x,t)+1).$
Since
$u^q\leq u+1,$
we conclude that
$$
	w_t-\Delta w = \exp{(mt)}\left(m(u+1)+a(x,t)u^r\int_\Omega{u^p(y,t)}\,dy-b(x,t)u^q\right) \geq 0
$$
in $Q_T.$ Arguing as in previous case, we obtain
$$
	u(x,t)\geq c_1\exp{(-mt)}y(x,t)-1\textrm{ в }\overline{\Omega_1}\times[0,T].
$$
Choosing
$
	c_1=(C+1)\exp{(mT)}\left(\inf_{\Omega_0\times(0,T)}{y(x,t)}\right)^{-1},
$
we have (\ref{110}).
\end{proof}
Now we prove that problem (\ref{v:u})--(\ref{v:n}) has
 finite time blow-up solutions if either $l>\max{(1,(q+1)/2)}$ and
    \begin{equation}\label{Gl:k(x,y,t)2}
    k(x,y,t)\geq k_0>0,\;x\in\partial\Omega,\;y\in\Omega,\;0<t<t_0,
    \end{equation}
for some positive constants $k_0$ and $t_0$  or
$r+p>\max(q,1)$ and
\begin{equation}\label{Gl:a(x,t)2}
a(x,t)\geq a_0>0,\;x\in\Omega,\;0<t<t_1,
\end{equation}
for some positive constants $a_0$ and $t_1.$
\begin{theorem}
There exist finite time blow-up solutions of~(\ref{v:u})--(\ref{v:n}) if either $l>\max{(1,(q+1)/2)}$
and (\ref{Gl:k(x,y,t)2}) holds or $r+p>\max(q,1)$  and (\ref{Gl:a(x,t)2}) holds.
\end{theorem}
\begin{proof}
We suppose at first that $l>\max{(1,(q+1)/2)}$ and (\ref{Gl:k(x,y,t)2}) holds. Let us consider the following problem
    \begin{eqnarray} \label{Gl:ad_problem2}
    \left\{
    \begin{array}{ll}
    u_t=\Delta u-b(x,t)u^q,\;x\in\Omega,\;t>0,\\
    u(x,t)=\int_{\Omega}{k(x,y,t)u^l(y,t)}\,dy,\;x\in\partial\Omega,\;t>0,\\
    u(x,0)= u_0(x),\; x\in\Omega.
    \end{array} \right.
    \end{eqnarray}
As it is proved in \cite{GG1}, problem~(\ref{Gl:ad_problem2}) has positive finite time blow-up solutions.
We note that any solution of~(\ref{Gl:ad_problem2}) is a subsolution of~(\ref{v:u})--(\ref{v:n}).
Applying Theorem~\ref{Gl:theorem:comp-prins}, we prove the theorem.

Now we assume that $r+p>\max(q,1)$ and (\ref{Gl:a(x,t)2}) holds. We put $\overline b=\sup\limits_{Q_{t_1}}{b(x,t)}.$

Let $r\geq q>1.$  We denote
\begin{equation}\label{Gl:J(t)}
J(t)=\exp{(\lambda_1 t)}\int_{\Omega}{u(x,t)\varphi(x)}\,dx,
\end{equation}
where $\varphi(x)$ is the solution of (\ref{Bl:phi}) satisfying
   \begin{equation}\label{varphi_int}
   \int_{\Omega}{\varphi(x)}\,dx=1.
   \end{equation}
 Then using (\ref{v:u}), (\ref{Bl:phi}), Green's identity and the
inequality
    \begin{equation}\label{Gl:varphi_int}
   \frac{\partial\varphi(x)}{\partial n} \leq 0, \quad x \in \partial
   \Omega,
    \end{equation}
where $\nu$ is unit outward normal on $\partial\Omega,$  we get for $t<t_1$
    \begin{equation}\label{Gl:J}
 \begin{split}
J'(t) \geq & \exp{(\lambda_1 t)} \int_{\Omega}{\left(a(x,t)u^r\int_{\Omega}{u^p(y,t)}\,dy-b(x,t)u^q\right)\varphi(x)}\,dx \\
\geq & \exp{(\lambda_1 t)}\left(a_0\int_{\Omega}{u^p(y,t)}\,dy\int_{\Omega}{u^r(x,t)\varphi(x)}\,dx-\overline b\int_{\Omega}{u^q(x,t)\varphi(x)}\,dx\right).
 \end{split}
\end{equation}
By Lemma~\ref{Gl:Lem2}
    \begin{equation}\label{Gl:lm1}
    a_0\int_{\Omega}{u^p(y,t)}\,dy\left[\int_{\Omega}{u^q(x,t)\varphi(x)}\,dx\right]^{\frac{r-q}{q}}-\overline b\geq c_2>0
    \end{equation}
for $t<t_1$ and large initial data. Taking into account~(\ref{Gl:J}), (\ref{Gl:lm1}),
H{$\ddot o$}lder's and Jensen's inequalities, we have for $t<t_1$
  $$
 	J'(t)\geq \exp{(\lambda_1 t)}\int_{\Omega}{u^q(x,t)\varphi(x)}\,dx\left(a_0\int_{\Omega}{u^p(y,t)}\,dy
 	\left[\int_{\Omega}{u^q(x,t)\varphi(x)}\,dx\right]^{\frac{r-q}{q}}-\overline b\right)
$$
$$
 	\geq c_2\exp{[\lambda_1(1-q)t]}J^q(t).
$$
Hence, $J(t)$ blows up in finite time $T\;(T<t_1)$ for large initial data.

Let $r>1\geq q.$ By Lemma~\ref{Gl:Lem2}
    \begin{equation}\label{Gl:lm2}
        \int_{\Omega}{u^p(y,t)}\,dy\geq c_3>0
    \end{equation}
for $t<t_1$ and large initial data. Then using
(\ref{Gl:J}), (\ref{Gl:lm2}), $u^q\leq u+1$ and Jensen's inequalities, we obtain for $t<t_1$
   $$
            J'(t) \geq  a_0 c_3 \exp{[\lambda_1(1-r)t]}J^r(t) - \overline bJ(t) - \overline b \exp{(\lambda_1 t)}
   $$
and again $J(t)$ blows up in finite time $T\;(T<t_1)$ for large initial data.

Let $r<q.$ Without loss of generality, we may assume
that $\Omega$ contains the origin.  We introduce designations
$$G=\left\lbrace(x,t):0\leq t<T_1,\;|x|\leq A\sqrt{T-t}\right\rbrace\;(T_1<T),\quad G_{\tau} = G \cap \{ t= \tau \} \;(0< \tau <T_1)$$
and consider the auxiliary problem
\begin{eqnarray} \label{Dir}
 \left\{
 \begin{array}{ll}
   u_t=\Delta u + a(x,t)u^r\int_{G_t} {u^p(y,t)}\,dy - b(x,t)u^q, \; (x,t) \in G, \\
   u(x,t)=0, \; |x| = A\sqrt{T-t}, \; 0<t<T_1,\\
    u(x,0)=  A^2-|x|^2/T,\; |x| < A\sqrt{T},
    \end{array} \right.
    \end{eqnarray}
where $ A>0 $ will be determined below and $T<\min
\{1, t_1\}$ we choose in such a way that points $x$ with $|x|\leq A\sqrt{T}$ belong to $\Omega.$
Let $u(x,t)$ be a positive in $\overline G$ solution of~(\ref{v:u})--(\ref{v:n}) such that
 $u_0 (x) \geq \left(A^2-|x|^2/T \right)_+.$ Obviously, $u(x,t)$ is a supersolution
of~(\ref{Dir}). We construct a subsolution of~(\ref{Dir}) in the following form
    \begin{equation}\label{Gl:V(A,T)}
        \underline u(x,t)=(T-t)^{-\gamma}V\left(\frac{|x|}{\sqrt{T-t}}\right),
    \end{equation}
where $V(\xi)=\left(A^2-\xi^2\right)_+,$ $\xi=|x|/\sqrt{T-t}$
and $\gamma >0$ will be chosen below. It is easy to see
$\underline u(0,t) \to \infty$ as $t \to T_1$ and $T_1 \to T.$
We show that
    \begin{equation}\label{Gl:equation}
    \Lambda \underline u \leq 0
    \end{equation}
in $G,$ where
    $$
    \Lambda u \equiv u_t - \Delta u - a(x,t)u^r\int_{G_t} {u^p(y,t)}\,dy + b(x,t)u^q.
    $$
Note that
    \begin{equation}\label{Gl:10}
    \int_{G_t}{V^p(\xi)}\,dx=(T-t)^{\frac{n}{2}}\int_{|z| \leq A}{(A^2-|z|^2)^p_+}\,dz = C(A) (T-t)^{\frac{n}{2}}.
    \end{equation}
By~(\ref{Gl:10}) we obtain
    \begin{eqnarray}\label{Gl:11}
        \Lambda \underline u
        &\leq& \gamma(T-t)^{-\gamma-1}V(\xi)-(T-t)^{-\gamma-1}\xi^2 +  2n(T-t)^{-\gamma-1} \nonumber \\
        &-& a_0 C(A)(T-t)^{\frac{n}{2}-\gamma(r+p)}V^r(\xi) + \overline b (T-t)^{-\gamma q} V^q(\xi)
    \end{eqnarray}
 for points of $G.$
 Further we distinguish the two zones  $0\leq\xi<\theta A$
and $\theta A\leq\xi<A,$ where $\theta\in(0,1)$ will be chosen below.

For $\theta A\leq\xi<A$  we have
    \begin{equation}\label{V}
    V(\xi) \leq (1-\theta^2) A^2.
    \end{equation}
From~(\ref{Gl:11}), (\ref{V}) it follows that
    \begin{eqnarray}\label{Gl:12}
     \Lambda \underline u & \leq & \left(\gamma(1-\theta^2)A^2-\theta^2A^2+2n\right)(T-t)^{-\gamma-1} - a_0 C(A) (T-t)^{\frac{n}{2}-\gamma(r+p)}V^r(\xi)  \nonumber \\
     &+& \overline b(T-t)^{-\gamma q}V^q(\xi).
    \end{eqnarray}
We put
       \begin{eqnarray}\label{A}
        A=3\sqrt{n},\;\theta^2=\frac{\gamma+1/2}{\gamma+1}
    \end{eqnarray}
and estimate the first term on the right hand side of~(\ref{Gl:12})
    \begin{equation}\label{Gl:13}
    \left(\gamma(1-\theta^2)A^2-\theta^2A^2+2n\right)(T-t)^{-\gamma-1} = -\frac{5n}{2}(T-t)^{-\gamma-1}<0.
    \end{equation}
By (\ref{V})
        \begin{equation}\label{Gl:14}
        a_0 C(A) (T-t)^{\frac{n}{2}-\gamma(r+p)}V^r(\xi) \geq \overline b (T-t)^{-\gamma q}V^q(\xi)
    \end{equation}
for small values of $T$ and $\gamma > n/[2(r+p-q)].$ From~(\ref{Gl:12}),
(\ref{Gl:13}), (\ref{Gl:14}) it follows~(\ref{Gl:equation}) for $\xi \in [\theta A, A).$

For $0\leq\xi<\theta A$ we have
$$V(\xi) \geq (1-\theta^2)A^2=\frac{9n}{2(\gamma+1)}.$$
Then by (\ref{Gl:11})  inequality (\ref{Gl:equation}) still holds for $ 0\leq\xi<\theta A$ if $T$ is small and
$$
\gamma>\max\left(\frac{n}{2(r+p-q)},\frac{n+2}{2(r+p-1)}\right).
$$
Applying comparison principle for (\ref{Dir}), we obtain
$u(x,t) \geq \underline u(x,t)$ in $G.$ Hence, $u(x,t)$ blows up in finite time.

In the case $q \leq r \leq 1$ we have $\gamma q<\gamma+1.$
Then the function in (\ref{Gl:V(A,T)}) satisfies (\ref{Gl:equation}) for $0 \leq \xi < A.$
Indeed, by virtue of (\ref{Gl:12})--(\ref{Gl:13}) we have
   $$
  \Lambda \underline u \leq \left(-\frac{5n}{2}(T-t)^{-\gamma-1}+\overline b(T-t)^{-\gamma q}V^q(\xi)\right)
    -      a_0 C(A) (T-t)^{\frac{n}{2}-\gamma(r+p)}V^r(\xi) \leq 0
    $$
  for $ \theta A \leq \xi < A$ and small values of $T.$
  For $0\leq\xi<\theta A$  inequality~(\ref{Gl:equation})
  holds if $\gamma> (n+2)/[2(r+p-1)]$ and $T$ is small.
Arguing as in the previous case, we complete the proof.
\end{proof}

%%%%%%%%%%%%%%%%%%%%%%%%%%%%%%%%%%%%%%%%%%%%%%%%%%%%%%%%%%%%%%%%%%%%%%%%%%%%%%%%%%%%%%%%%%%%%%%%%%%%%%%%%%%%%%%%%%

\section{Blow-up of all nontrivial solutions and global existence of solutions for small initial data}\label{Bl}
In this section we find conditions which guarantee blow-up in finite time
of all nontrivial solutions and prove global existence of solutions for small initial data.

First we show that for $q<\min(r+p,1)$ under some conditions problem (\ref{v:u})--(\ref{v:n})
has nontrivial global solutions for any $a(x,t)$ and $k(x,y,t).$  Suppose that
\begin{equation}\label{Bl:b(x,t)}
\inf\limits_{\Omega}{b(x,0)}>0.
\end{equation}
\begin{theorem}
Let (\ref{Bl:b(x,t)}) hold and either $q<\min(r+p,1),\;l>1$ or $r\geq q,$ $(q+1)/2<l\leq 1.$
 Then problem~(\ref{v:u})--(\ref{v:n}) has global solutions for small initial data.
\end{theorem}
\begin{proof}
We put $b_0=\inf\limits_{Q_T}{b(x,t)}$ and choose $T$ so that $b_0 > 0.$

Suppose that $q<\min(r+p,1),\;l>1.$  A straightforward
computation shows that  for small $\beta,$ $\varepsilon$ and $u_0(x)$
\begin{equation*}
    g(t)=\beta [T-t]_+^{\frac{1}{1-q}}+\varepsilon
\end{equation*}
is a supersolution of~(\ref{v:u})--(\ref{v:n}) in $Q_T.$
Applying Theorem~\ref{Gl:theorem:comp-prins}, we have $u(x,t)\leq g(t)$ in $Q_T$.
Passing to the limit as $\varepsilon\rightarrow0,$ we obtain
$$
u(x,t) \leq \beta [T-t]_+^{\frac{1}{1-q}},
\; (x,t) \in Q_T.
$$
Now we put  $u(x,t) \equiv 0$ for $t \geq T.$

For $r\geq q,$ $(q+1)/2<l \leq 1$ to construct a supersolution we use the
change of variables as in Theorem~\ref{Gl:theorem1}. For points of $Q_{\delta,t_0}$ define
     \begin{equation*}
    v(x,t)=  v((\overline x,s),t)=(\delta-s-t)^{\gamma}_+ +\varepsilon,
     \end{equation*}
where $\delta > 0, $ $\varepsilon> 0,$ $t_0< \min (\delta, T),$ $2/(1-q)<\gamma< 1/(1-l)$
for $l<1$ and $2/(1-q)<\gamma $ for $l=1.$ In $\overline{Q_{t_0}}\setminus Q_{\delta,t_0}$
we put $ v(x,t)= \varepsilon.$ Then $ v(x,t)$ is a supersolution of~(\ref{v:u})--(\ref{v:n}) in $Q_{t_0}$ if
$ u_0(x)\leq (\delta-s)^{\gamma}_+.$ Indeed, by (\ref{Gl:M}), (\ref{Gl:new-coord}),
(\ref{Gl:enq4}), (\ref{L})--(\ref{Est}) for small $\delta$ and $\varepsilon$
we have $ Lv \ge 0$ in $ Q_{t_0} \setminus Q_{\delta,t_0}$ and
$$
 	Lv\geq-\gamma(\delta-s-t)_+^{\gamma-1}-\gamma(\gamma-1)(\delta-s-t)_+^{\gamma-2}-\gamma\overline c(\delta-s-t)_+^{\gamma-1}
$$
$$
 	- 2^p M \left((\delta-s-t)_+^{\gamma} + \varepsilon\right)^r \left( \delta^{\gamma p + 1} \overline J + \varepsilon^p |\Omega| \right) +
 	b_0\left((\delta-s-t)_+^{\gamma}+\varepsilon\right)^q \geq 0 \textrm{ in }Q_{\delta,t_0}.
$$
Estimating the integral $I$ in right hand side of (\ref{Gl:boundary}), we obtain
    \begin{equation*}
        I \leq M \left(
        \int_{\Omega} \left(\delta-s-t\right)_+^{\gamma l}\,dy +
         |\Omega| \varepsilon^l \right) \leq M \left(\overline J \frac{(\delta-t)_+^{\gamma l+1}}{\gamma l+1} + |\Omega|\varepsilon^l\right).
    \end{equation*}
On the other hand, we have $v((\overline x,0),t) = (\delta-t)^{\gamma}_+ + \varepsilon$
and~(\ref{Gl:boundary}) holds for
$$
\delta < \left( \frac{\gamma l+1}{2 M \overline J}
\right)^\frac{1}{\gamma (l-1)+1}, \quad \varepsilon < \left \{
\frac{(\delta-t_0)^\gamma}{2M|\Omega|} \right \}^\frac{1}{l}.
$$
By Theorem~\ref{Gl:theorem:comp-prins} $u(x,t)\leq v(x,t)$ in
$\overline{Q_{t_0}}$ and passing to the limit as $t_0 \to \delta$ and $\varepsilon \to 0,$ we deduce
    $$
		u(x,t)\leq (\delta-s-t)^{\gamma}_+  \textrm{ in } Q_\delta.
    $$
We put $u(x,t) \equiv 0$ for $t \geq \delta$ and complete the proof.
\end{proof}

Now suppose that $q=1.$ We set
\begin{equation}\label{Def:a,b}
\overline a(t)=\sup\limits_\Omega{a(x,t)}, \underline
a(t)=\inf\limits_\Omega{a(x,t)}, \overline
b(t)=\sup\limits_\Omega{b(x,t)}, \underline
b(t)=\inf\limits_\Omega{b(x,t)}, \underline
k(t)=\inf\limits_{\partial\Omega\times\Omega}{k(x,y,t)}.
\end{equation}
Problem (\ref{v:u})--(\ref{v:n}) has global solutions for small initial data if
$q=1,\,\min(r+p,l)>1$ and
\begin{equation}\label{Bl:a(x,t)_v}
    \int_0^{\infty}{\overline a(t)\exp{\left[-(r+p-1)\left(\sigma t+\int_0^t{\underline b(\tau)}\,d\tau\right)\right]}}\,dt<\infty,\;\sigma<\lambda_1,
\end{equation}
\begin{equation}\label{Bl:k(x,y,t)_l}
\int_{\Omega}{k(x,y,t)}\,dy\leq K \exp{\left[(l-1)\left(\gamma
t+\int_0^t{\underline b(\tau)}\,d\tau\right)\right]},\;
x\in\partial\Omega, \; t>0,\;K>0,\;\gamma<\lambda_1
\end{equation}
and conversely (\ref{v:u})--(\ref{v:n}) has no global nontrivial solutions if either $q=1,$ $\min(r,p) \geq 1$ and
\begin{equation}\label{Bl:a(x,t)_v2}
\int_0^{\infty}{\underline
a(t)\exp{\left[-(r+p-1)\left(\lambda_1t+\int_0^t{\overline
b(\tau)}\,d\tau\right)\right]}}\,dt = \infty,
\end{equation}
or $q=1,\,l>1$ and
\begin{equation}\label{Bl:k(x,y,t)_l2}
\int_0^{\infty}{\underline
k(t)\exp{\left[-(l-1)\left(\lambda_1t+\int_0^t{\overline
b(\tau)}\,d\tau\right)\right]}}\,dt = \infty.
\end{equation}
\begin{theorem}\label{q=1}
Let $q=1,\,\min(r+p,l)>1$ and~(\ref{Bl:a(x,t)_v}), (\ref{Bl:k(x,y,t)_l}) hold.
Then there exist global solutions of~(\ref{v:u})--(\ref{v:n}) for small initial data.
If either $q=1,\,\min(r,p)\geq1$ and (\ref{Bl:a(x,t)_v2}) holds
or $q=1,\,l>1$ and (\ref{Bl:k(x,y,t)_l2}) holds, then any nontrivial solution of~(\ref{v:u})--(\ref{v:n}) blows up in finite time.
\end{theorem}
\begin{proof}
Assume that $T$ is any positive constant, $q=1,\,\min(r+p,l)>1$ and (\ref{Bl:a(x,t)_v}), (\ref{Bl:k(x,y,t)_l}) hold.
We choose $\widetilde{\lambda}$ in such a way that
$$\max(\sigma,\gamma)<\widetilde{\lambda}<\lambda_1.$$
Let $\widetilde{\Omega}$ be bounded domain in $\mathbb{R}^n$
with smooth boundary such that $\Omega\subset\subset\widetilde{\Omega}$
and $\widetilde{\lambda}$ be the first eigenvalue of (\ref{Bl:phi})
in $\widetilde{\Omega}.$ Then correspondent eigenfunction $\widetilde\varphi(x)$ satisfies
    $$
    \frac{\sup_{\widetilde\Omega}\widetilde\varphi(x)}{\inf_{\Omega}\widetilde\varphi(x)}<d
    $$
for some $d>0.$ Choosing
    $$
    0<\varepsilon\leq\left(K d^l\right)^{-\frac{1}{l-1}},\;\sup\limits_{\widetilde{\Omega}}{\widetilde\varphi(x) }=d\varepsilon,
    $$
we have $\inf\limits_{\partial\Omega}{\widetilde\varphi(x)}>\varepsilon.$
We put  $N=\sup\limits_\Omega{\widetilde\varphi^{r-1}(x) \int_\Omega{\widetilde\varphi^p(y)}\,dy}$ and
$$
f(t)=\exp{(-\widetilde{\lambda}t)}\left(B-(r+p-1)N\int_0^{t}{\overline a(\tau)
\exp{\left[-(r+p-1)\left(\widetilde\lambda\tau+\int_0^\tau{\underline
b(s)}\,ds\right)\right]}}\,d\tau\right)^{-\frac{1}{r+p-1}},
$$
where $$
B=1+(r+p-1)N\int_0^{\infty}{\overline a(\tau)
\exp{\left[-(r+p-1)\left(\widetilde\lambda\tau+\int_0^\tau{\underline
b(s)}\,ds\right)\right]}}\,d\tau.
$$

It is easy to check that
$$
v(x,t)=\widetilde\varphi(x)f(t)\exp{\left(-\int\limits_0^t{\underline b(\tau)}\,d\tau\right)}
$$
is a supersolution of~(\ref{v:u})--(\ref{v:n}) in $Q_T$  for $u_0(x) \leq B^{-\frac{1}{r+p-1}}\widetilde\varphi(x).$
By Theorem~\ref{Gl:theorem:comp-prins} there exist global solutions of
(\ref{v:u})--(\ref{v:n}).

Now suppose that $q=1,\,\min(r,p)\geq1$ and (\ref{Bl:a(x,t)_v2}) holds.
Multiplying (\ref{v:u}) by $\varphi(x)\exp(\lambda_1t),$
where $\varphi(x)$ is defined in (\ref{Bl:phi}) and (\ref{varphi_int}),
and integrating the obtained equation over $\Omega,$ from (\ref{Gl:J(t)}), (\ref{Gl:varphi_int}),
Green's identity and Jensen's inequality, we obtain
    $$
    J'(t)\geq [\sup_\Omega \varphi(x)]^{-1} \underline a(t)\exp[\lambda_1(1-r-p)t]J^{r+p}(t) - \overline b(t)J(t).
    $$
Now (\ref{Bl:a(x,t)_v2}) guarantees blow-up of $J(t)$ in finite time. The case $q=1,\,l>1$ is treated similarly.
\end{proof}
\begin{remark}
The conclusion of Theorem~\ref{q=1} is not true if $\sigma > \lambda_1$ in (\ref{Bl:a(x,t)_v}),
or $\gamma=\lambda_1$ in (\ref{Bl:k(x,y,t)_l}), or $\lambda_1$ is replaced by a smaller value
in~(\ref{Bl:a(x,t)_v2}) or (\ref{Bl:k(x,y,t)_l2}).
\end{remark}

Further we consider the case $q>1.$ To prove blow-up of all nontrivial solutions
we need universal lower bound for solutions of (\ref{v:u})--(\ref{v:n}). Assume that
\begin{equation}\label{Bl:u}
b(x,t)\leq \varepsilon(t)\exp[\lambda_1 (q-1)t],\; x\in\Omega, \;
t>0,
\end{equation}
where
\begin{equation}\label{Bl:Eps}
\varepsilon(t) \in C([0,\infty)),\; \varepsilon(t) \geq 0, \;
\int_0^\infty \varepsilon(t) \, dt < \infty.
\end{equation}
\begin{lemma}\label{Bl:Est}
Let $u(x,t)$ be a solution of (\ref{v:u})--(\ref{v:n})
in $Q_T,$ $q>1,$ $u_0(x)\not\equiv 0$ and (\ref{Bl:u}), (\ref{Bl:Eps}) hold. Then for any
$t_0 \in (0,T)$ there exists $d>0,$ which does not depend on $T,$ such that
    \begin{equation}\label{Bl:low}
u(x,t)\geq d  \varphi(x)\exp(-\lambda_1 t),\; x\in\Omega, \; t \in
(t_0, T),
\end{equation}
where $\varphi(x)$ is defined in (\ref{Bl:phi}) and
(\ref{varphi_int}).
\end{lemma}
\begin{proof}
For $T_0 \in (0,T)$ denote $m_0=\max\left(\sup\limits_{Q_{T_0}}{u(x,t)},\sup\limits_{Q_{T_0}}{b(x,t)}\right).$
Let $v(x,t)$ be a solution of the problem
\begin{eqnarray} \label{problemV}
    \left\{
    \begin{array}{ll}
    v_t=\Delta v - b(x,t) v^q,\;(x,t) \in Q_{T_0},\\
    v(x,t)=0,\; (x,t) \in S_{T_0},\\
    v(x,0)= v_0(x),\; x\in\Omega,
    \end{array} \right.
    \end{eqnarray}
where $v_0(x) \in C_0^\infty (\Omega),\,$ $0 \leq v_0(x) \leq u_0(x)
\,$ and $v_0(x) \not\equiv 0.$ Obviously, $u(x,t)$ is a supersolution of~(\ref{problemV}).
By comparison principle for (\ref{problemV}) we obtain
\begin{equation}\label{L1}
u(x,t) \geq v(x,t), (x,t) \in Q_{T_0}.
\end{equation}
We put $h(x,t) = \exp(\mu t) v(x,t),$ where $\mu \geq m_0^q.$ Then in $Q_{T_0}$
$$
h_t - \Delta h \geq \exp(\mu t) v (\mu - b(x,t) v^{q-1}) \geq 0.
$$
Since $h(x,0)= v_0(x)$ and $v_0(x)$ is nontrivial nonnegative function in $\Omega,$ by strong maximum principle
\begin{equation}\label{L2}
h(x,t)> 0, \; (x,t) \in Q_{T_0}.
\end{equation}
By virtue of Theorem~3.6 in \cite{H}
\begin{equation}\label{L3}
\max_{\partial\Omega}\frac{\partial h(x,t_0)}{\partial n} < 0,
\end{equation}
where $t_0 \in (0,T_0).$ From (\ref{L2}) and (\ref{L3}) it follows that
$$
v(x,t)> 0 \; \textrm{ in } \; Q_{T_0} \; \textrm{ and } \;
\max_{\partial\Omega}\frac{\partial v(x,t_0)}{\partial n} < 0.
$$
Then there exists positive constant $d_0$ such that
 \begin{equation}\label{L4}
v(x,t_0)\geq d_0  \varphi(x)\exp(-\lambda_1 t_0),\; x \in\Omega.
\end{equation}
By (\ref{L1}) and (\ref{L4})
 \begin{equation*}\label{}
u(x,t_0)\geq d_0  \varphi(x)\exp(-\lambda_1 t_0),\; x \in\Omega.
\end{equation*}
 A straightforward
computation shows that for large $f_0$
$$
\underline u (x,t) = \varphi(x) \exp{(-\lambda_1 t)} \left\{ f_0 + (q-1) [\sup_\Omega \varphi(x)]^{q-1} \int_{t_0}^{t} \varepsilon(\tau) d \tau \right\}^{-\frac{1}{q-1}}
$$
is a subsolution of~(\ref{problemV}) in $Q_{T_0}
\setminus \overline{Q_{t_0}}$ with initial datum $v(x,t_0) =
u(x,t_0).$ Application of comparison principle for (\ref{problemV}) completes the proof.
\end{proof}
Next we assume that
\begin{equation}\label{Bl0:k}
\int_{\Omega}{k(x,y,t)}\,dy\leq A\exp{(\sigma t)},\;
x\in\partial\Omega, \; t>0, \; A>0,\;\sigma<\lambda_1(l-1)
\end{equation}
and
\begin{equation}\label{Bl0:b1}
b(x,t)\geq Ba(x,t)\exp{(-\omega t)},\; x\in\Omega, \; t>0, \;
B>0,\;\omega<\lambda_1(r+p-q)
\end{equation}
or $b(x,t)$ satisfies (\ref{Bl:u}), (\ref{Bl:Eps}), where
\begin{equation}\label{Bl0:b}
\lim_{t \to \infty}  \varepsilon(t) = 0,
\end{equation}
 and
\begin{equation}\label{Bl0:k2}
k(x,y,t)\geq D\exp{[\lambda_1(l-1)t]},\; x\in\partial\Omega,\;
y\in\Omega, \; D>0
\end{equation}
for large values of $t.$
\begin{theorem}\label{q>1}
If $l>1,\;1<q<r+p$ and (\ref{Bl0:k}), (\ref{Bl0:b1}) hold,
then there exist global solutions of (\ref{v:u})--(\ref{v:n}) for small initial data.
If $l\geq q>1$ and (\ref{Bl:u}), (\ref{Bl:Eps}), (\ref{Bl0:b}),
(\ref{Bl0:k2}) hold, then any nontrivial solution of (\ref{v:u})--(\ref{v:n}) blows up in finite time.
\end{theorem}
\begin{proof}
Let $l>1,\;1<q<r+p$ and (\ref{Bl0:k}), (\ref{Bl0:b1}) hold.
We choose $\widetilde{\lambda}_1$ in the following way
$$\max\left(\omega/(r+p-q),\sigma/(l-1)\right)<\widetilde{\lambda}_1<\lambda_1.$$
Let $\widetilde{\Omega}$ be bounded domain in $\mathbb{R}^n$
with smooth boundary such that $\Omega\subset\subset\widetilde{\Omega}$
and $\widetilde{\lambda}_1$ be the first eigenvalue of (\ref{Bl:phi})
in $\widetilde{\Omega}.$ The correspondent eigenfunction $\widetilde\varphi(x)$ is chosen to satisfy that
$\sup_{\widetilde{\Omega}}\widetilde\varphi(x) = 1.$
Obviously, $\inf_{\Omega}\widetilde\varphi(x) > d$ for some $d>0.$
Then $v(x,t)=\beta\exp{(-\widetilde{\lambda}_1t)}\widetilde{\varphi}(x)$ is a supersolution of
(\ref{v:u})--(\ref{v:n}) in $Q_T$ for any $T>0$ if

 $$
 \beta\leq\min\left(\left(\frac{B}{\sup_{\Omega}{\widetilde\varphi^{r-q}(x)}
    \int_\Omega{\widetilde\varphi^p(y)}\,dy}\right)^{\frac{1}{r+p-q}},\,
    \left(\frac{d}{A}\right)^{\frac{1}{l-1}}\right),\; u_0(x)\leq \beta\widetilde\varphi(x).
 $$
By Theorem~\ref{Gl:theorem:comp-prins} there exist global solutions of
(\ref{v:u})--(\ref{v:n}).

 Let $u(x,t)$ be nontrivial global solution of (\ref{v:u})--(\ref{v:n}), $l\geq q>1$
 and (\ref{Bl:u}), (\ref{Bl:Eps}), (\ref{Bl0:b}), (\ref{Bl0:k2}) hold. Then by (\ref{v:g}),
 (\ref{Bl:low}) and (\ref{Bl0:k2}) there exist positive constants $t_1$ and $d_1$ such that
\begin{equation}\label{b1}
        u(x,t) \geq d_1  \exp(-\lambda_1 t),\;x\in\partial\Omega,\;t \geq t_1.
        \end{equation}
Let us consider auxiliary problem
\begin{eqnarray} \label{problemV2}
    \left\{
    \begin{array}{ll}
    v_t = \Delta v - b(x,t) v^q, \; x\in\Omega,\;t>t_2,\\
    v(x,t) = u(x,t),\; x\in\partial\Omega,\;t>t_2,\\
    v(x,t_2) = u (x,t_2),\; x\in\Omega,
    \end{array} \right.
    \end{eqnarray}
where $t_2 \geq t_1.$ Using  (\ref{Bl:u}), (\ref{Bl0:b}),
(\ref{b1}), we check that  $\underline u (x,t) = d_2  \exp(-\lambda_1
t)$  is a subsolution of (\ref{problemV2}) under suitable choice of $t_2$ and $d_2 > 0.$
Comparison principle for (\ref{problemV2}) gives
\begin{equation}\label{b2}
u(x,t) \geq d_2  \exp(-\lambda_1 t),\;x \in \Omega,\;t \geq
t_2.
\end{equation}
Let  $\varphi(x)$ satify (\ref{Bl:phi}) and (\ref{varphi_int}). Multiplying (\ref{v:u}) by $\varphi(x)\exp(\lambda_1t),$
integrating over $\Omega$ and using
$$
\int_{\partial\Omega} \frac{\partial\varphi(x)}{\partial n} \,ds = - \lambda_1,
$$
Green's identity, Jensen's inequality and
(\ref{Gl:J(t)}), (\ref{Gl:varphi_int}), (\ref{Bl:u}), (\ref{Bl0:b})--(\ref{b1}),
(\ref{b2}), we obtain
 $$
	J'(t) \geq \int_{\Omega} \left( \lambda_1 [\sup_\Omega \varphi(x)]^{-1} D\exp[\lambda_1lt] u^{l-q} - \varepsilon(t) \exp[\lambda_1qt]  \right)  u^q  \varphi(x)  dx\geq  d_3 J^q (t)
$$
for some $d_3 >0$ and large values of $t>0.$
Integrating differential inequality, we prove the theorem.
\end{proof}
\begin{remark}
 Theorem~\ref{q>1} does not hold if $\sigma=\lambda_1(l-1)$
in (\ref{Bl0:k}) or $\lambda_1$ is replaced by a smaller value in (\ref{Bl0:k2}).
Furthermore, we can not $\varepsilon (t)$ replace by any positive constant in (\ref{Bl:u}).
Indeed, let $a(x,t) \equiv 0,\,$ $b(x,t) = b\exp[\lambda_1 (q-1)t], \,$ $k(x,y,t) = k \exp{[\lambda_1(l-1)t]},$
 where $b$  and   $k$ are positive constants. Then
$$
\overline u(x,t) = \left\{ \frac{\lambda_1}{b} \right\}^\frac{1}{q-1}  \exp (-\lambda_1 t)
$$
is a supersolution of (\ref{v:u})--(\ref{v:n}) if
$\min (q,l) > 1, \,$
$$
k \leq  \frac{1}{|\Omega|} \left\{ \frac{b}{\lambda_1} \right\}^\frac{l-1}{q-1} \,\,\, \textrm{ and } \,\,\, u_0 (x) \leq \left\{ \frac{\lambda_1}{b} \right\}^\frac{1}{q-1}.
$$
By Theorem~\ref{Gl:theorem:comp-prins} there exist global solutions of
(\ref{v:u})--(\ref{v:n}).
\end{remark}

To prove blow-up of all nontrivial solutions of (\ref{v:u})--(\ref{v:n}) for $\max (r,p) \geq q >1$
we assume that
 \begin{equation}\label{a1}
\underline a(t) = \gamma (t) \exp\left[\lambda_1 (r+p-q) t \right]
\overline b(t),
\end{equation}
 \begin{equation}\label{a2}
\int_{0}^{\infty} \underline a(t) \exp\left[-\lambda_1 (r+p-1) t
\right] \, dt = \infty,
\end{equation}
where $\lim_{t \to \infty} \gamma (t) = \infty,$
$\underline a(t)$ and $\overline b(t)$ are defined in (\ref{Def:a,b}).
\begin{theorem}\label{r,p>q>1}
Let $\max (r,p) \geq q >1$ and (\ref{Bl:u}),
(\ref{Bl:Eps}), (\ref{a1}), (\ref{a2}) hold.
Then any nontrivial solution of~(\ref{v:u})--(\ref{v:n}) blows up in finite time.
\end{theorem}
\begin{proof}
Denote
 \begin{equation}\label{I}
I(t) = \int_{\Omega} \left\{ \frac{1}{2} \underline a(t) u^r (x,t)
\int_{\Omega} u^p (y,t) \, dy -  \overline b(t) u^q (x,t) \right
\} \varphi(x) \, dx,
\end{equation}
where $\varphi(x)$ is defined in (\ref{Bl:phi}) and
(\ref{varphi_int}). Suppose at first that $p \geq q.$
By (\ref{Bl:low}), (\ref{a1}), (\ref{I}) and H{$\ddot o$}lder's inequality it follows
 $$
	I(t) \geq \int_{\Omega} \left\{ \frac{1}{2} \underline a(t) u^r
	 [\sup_\Omega \varphi(x)]^{-1}
	\int_{\Omega} u^p (y,t) \varphi(y) \, dy -  \overline b(t) u^q
	\right\} \varphi(x) \, dx
$$
$$
	\geq \overline b(t) \left[ \int_{\Omega}
	u^p \varphi \, dx  \right]^\frac{q}{p} \left\{ \frac{d^r}{2}
	[\sup_\Omega \varphi(x) ]^{-1} \gamma (t)  \exp\left[\lambda_1
	(p-q) t \right] \int_{\Omega} \varphi^{r+1}\, dx \left[
	\int_{\Omega} u^p  \varphi \, dy
	\right]^\frac{p-q}{p}  - 1 \right\}
$$
 \begin{equation}\label{I1}
	\geq \overline b(t) \left[ \int_{\Omega}
	u^p \varphi \, dx  \right]^\frac{q}{p} \left\{ \frac{d^{r+p-q}}{2}
	[\sup_\Omega \varphi(x)]^{-1}\gamma (t)  \int_{\Omega}
	\varphi^{r+1}\, dx \left[ \int_{\Omega} \varphi^{p+1} \, dy
	\right]^\frac{p-q}{p}  - 1 \right\} \geq 0
\end{equation}
for large values of $t.$ Using H{$\ddot o$}lder's inequality again and
(\ref{Gl:J}), (\ref{Bl:low}), (\ref{I1}), we obtain
$$
	J'(t) \geq\exp (\lambda_1  t ) \int_{\Omega} \left\{ \underline
	a(t) u^r  [\sup_\Omega \varphi(x)]^{-1} \int_{\Omega} u^p
	(y,t) \varphi(y) \, dy - \overline b(t) u^q
	\right\} \varphi(x) \, dx
$$
$$
	\geq\frac{1}{2} [\sup_\Omega \varphi(x)]^{-1} \exp (\lambda_1  t ) \underline
	a(t)  \int_{\Omega} u^r (x,t) \varphi(x) \, dx \int_{\Omega} u^p
	(y,t) \varphi(y) \, dy
$$
 \begin{equation}\label{J1}
	\geq\frac{d^r}{2} [\sup_\Omega \varphi(x)]^{-1} \int_{\Omega} \varphi^{r+1}\, dx  \exp\left[-\lambda_1 (r+p-1) t \right]
	\underline a(t) J^p (t)
\end{equation}
for large values of $t.$  By (\ref{a2}), (\ref{J1}) $J(t)$ blows up in finite time.

Suppose that $r \geq q.$ Arguing as in previous case, we obtain
 \begin{equation}\label{J2}
J'(t) \geq \frac{d^p}{2} \int_{\Omega} \varphi^{p}\, dx
\exp\left[-\lambda_1 (r+p-1) t \right] \underline a(t) J^r (t)
\end{equation}
for large values of $t$ and $J(t)$ blows up in finite time again.
\end{proof}

\begin{remark}
Theorem~\ref{r,p>q>1} does not hold if $\gamma (t)$ is a bounded function in (\ref{a1}).
Indeed, suppose that $p \geq 1,\,$ $r \geq q > 1,\,$ $a(x,t) \equiv \underline
 a(t),\,$ $b(x,t) \equiv \overline b(t)\,$ and $k(x,y,t) \equiv 0.$
Let $\varphi(x)$ be defined in (\ref{Bl:phi}) and
(\ref{varphi_int}). Then $ \overline u(x,t) = \beta \varphi (x) \exp (-\lambda_1 t )$ is a supersolution of
(\ref{v:u})--(\ref{v:n}) for $u_0 (x) \leq \beta \varphi (x),\;x \in \Omega$ and small
 $\beta >0.$ By Theorem~\ref{Gl:theorem:comp-prins} there exist global solutions of
(\ref{v:u})--(\ref{v:n}).
\end{remark}

\begin{remark}
We note here an importance of divergence of the integral in (\ref{a2}) for blow-up
of all nontrivial solutions of (\ref{v:u})--(\ref{v:n}). Suppose that $a(x,t) \equiv \underline
 a(t)\,$ and $k(x,y,t) \equiv 0.$ Let $\varphi(x)$ be the solution of (\ref{Bl:phi})
 with $\sup_{\Omega}{\varphi(x)} = 1$ and
 $f(t)$ be a solution of the following differential equation
 $$
 f'(t) + \lambda_1 f(t) - |\Omega| \underline a(t) f^{r+p} (t) = 0.
 $$
Since the integral in (\ref{a2}) converges, $f(t)$ exists for any $t \geq 0$
if $f(0)$ is small enough. Then
$\overline u(x,t) = \varphi (x) f(t)$ is a supersolution of
(\ref{v:u})--(\ref{v:n}) with $u_0 (x) \leq \varphi (x)f(0).$
By Theorem~\ref{Gl:theorem:comp-prins} there exist global solutions of
(\ref{v:u})--(\ref{v:n}).
\end{remark}

\begin{remark}
Theorem~\ref{r,p>q>1} is not true if $\lambda_1$ is replaced by a larger value in  (\ref{Bl:u}).
Indeed, let $\sigma>0.$ We put
$$
a(x,t) = \frac{\lambda_1 (q-1) + \sigma}{(q-1)|\Omega|} \exp
\left[ \left( r+p-1 \right) \left( \lambda_1 + \frac{\sigma}{q-1}
\right) t\right],
$$
$$
b(x,t) = 2 \left( \lambda_1 + \frac{\sigma}{q-1} \right)
\exp\{[\lambda_1 (q-1) + \sigma] t \}, \;
$$
$$
k(x,y,t) = \frac{1}{|\Omega|} \exp \left[ ( l-1 ) \left( \lambda_1
+ \frac{\sigma}{q-1} \right) t\right].
$$

It is easy to see that
$$
 u(x,t) = \exp \left[ - \left( \lambda_1
+ \frac{\sigma}{q-1} \right) t\right]
$$
is the solution of~(\ref{v:u})--(\ref{v:n}) with
$u_0 (x) \equiv 1,\;x \in \Omega.$
\end{remark}

\begin{remark}
From Theorem~\ref{r,p>q>1} it follows that
Theorem~\ref{q>1} does not hold for $\omega > \lambda_1(r+p-q)$ in (\ref{Bl0:b1}).
\end{remark}

%%%%%%%%%%%%%%%%%%%%%%%%%%%%%%%%%%%%%%%%%%%%%%%%%%%%%%%%%%%%%%%%% Литература

\end{document}